\documentclass[12pt,a4paper,reqno]{amsart}
  \usepackage[margin=2.5cm]{geometry}
  \usepackage[T1]{fontenc}
  \usepackage{lmodern}
  \usepackage{microtype}
  \usepackage{amsmath,amssymb,mathtools}
  \usepackage{enumitem}
  \usepackage{hyperref}

  \hypersetup{
    hidelinks,
    pdftitle={Large Corank of Dense Random Regular Digraphs}
  }

  \usepackage{cleveref}
    \crefname{theorem}{Theorem}{Theorems}
    \crefname{theorem}{Theorem}{Theorems}
    \crefname{theorem*}{Theorem*}{Theorems}
    \crefname{lemma}{Lemma}{Lemmas}
    \crefname{lemma}{Lemma}{Lemmas}
    \crefname{remark}{Remark}{Remarks}
    \crefname{proposition}{Proposition}{Propositions}
  \crefname{notation}{Notation}{Notations}
  \crefname{claim}{Claim}{Claims}
    \crefname{definition}{Definition}{Definitions}
    \crefname{corollary}{Corollary}{Corollaries}
    \crefname{section}{Section}{Sections}
    \crefname{figure}{Figure}{Figures}
      \crefname{assumption}{Assumption}{Assumptions}

  \newtheorem{theorem}{Theorem}[section]
  \newtheorem{theorem*}{Theorem*}[section]

  \newtheorem{lemma}[theorem]{Lemma}
  
  \newtheorem{proposition}[theorem]{Proposition}

  \numberwithin{equation}{section}

  \theoremstyle{definition}

  \usepackage{xcolor}

  \DeclareMathOperator{\rank}{rank}
  \DeclareMathOperator{\corank}{corank}
  \DeclareMathOperator{\codim}{codim}
  \DeclareMathOperator{\supp}{supp}
  \DeclareMathOperator{\Ber}{Ber}
  \DeclareMathOperator{\Bin}{Bin}

  \newcommand{\R}{\mathbb R}
  \newcommand{\Q}{\mathbb Q}
  \newcommand{\Prob}{\mathbb P}
  \newcommand{\Mnd}{\mathcal M_{n,d}}
  
  \newcommand{\one}{\mathbf 1}

  \title[Large corank of dense random regular digraphs]
  {Large Corank of Dense Random Regular Digraphs}

  \vspace{1cm}

  \author{\textsc{Huaijin Liang, Kexin Yu, and Tingzhou Yu}}

  \vspace{0.3cm}

  \date{}

\begin{document}
  \raggedbottom
  \begin{abstract}
  Let $1\le k\le n$ and let $A$ be the adjacency matrix of a uniformly random $d$-regular directed graph on $n$ vertices. Suppose that $\lambda n\le d\le(1-\lambda)n$ for a fixed $0<\lambda\le1/2$. We show that there exists $c_{\lambda}>0$ depending only on $\lambda$ such that
  \[
   \Prob[\rank(A)\le n-k]\le 2e^{-c_{\lambda} nk}.
  \]
  This gives a large corank extension of the exponential singularity bound of Jain, Sah, and Sawhney.
  \end{abstract}


  \maketitle

  \section{Introduction}

  For positive integers $d\le n$, let $\Mnd$ denote the set of all
  $n\times n$ matrices with entries in $\{0,1\}$ whose row and column
  sums are equal to $d$, and let $A$ be uniformly distributed on
  $\Mnd$. Equivalently, $A$ is the adjacency matrix of a uniformly
  random $d$-regular directed graph (digraph) on $n$ labeled vertices, where
  self-loops are allowed but multiple edges are not, and every vertex
  has in-degree and out-degree $d$.
  We study the rank of $A$ over $\R$ in the range
  $\lambda n\le d\le(1-\lambda)n$ for fixed $0<\lambda\le1/2$. Our main result is the following.
  \begin{theorem}\label{thm:main}
  Let $0<\lambda\le1/2$. There exists $c_\lambda>0$, depending only on
  $\lambda$, such that for all positive integers $n,d$ satisfying
  $\lambda n\le d\le(1-\lambda)n$, if $A$ is uniform on $\Mnd$, then
  for every integer $1\le k\le n$,
  \[
   \Prob[\rank(A)\le n-k]\le 2e^{-c_\lambda nk}.
  \]
  \end{theorem}

  \subsection{Background}

  Large corank bounds go back to Kahn, Koml\'os, and Szemer\'edi in~\cite[Corollary~3]{KKS}. For an
  $n\times n$ matrix $\Xi$ with i.i.d.\ Rademacher entries, they
  proved that, for every $\varepsilon>0$, there exists an
  integer $C=C(\varepsilon)>0$ such that
  \[
   \Prob\bigl[\rank(\Xi)< n-C\bigr]\le \varepsilon^n.
  \]
  For $1\le k<n$, the event that $k+1$ specified rows of $\Xi$
  are identical has probability $2^{-nk}$ and implies
  $\rank(\Xi)\le n-k$. This suggests an upper bound of the form
  $\exp(-cnk)$. Rudelson in~\cite{Rudelson} proved that, for
  matrices $M$ with i.i.d.\ nonconstant subgaussian entries,
  \[
   \Prob\bigl[\rank(M)\le n-k\bigr]\le \exp(-c_1nk),
   \qquad 1\le k\le c_2\sqrt n,
  \]
  where $c_1,c_2>0$ depend on the entry distribution.
  For i.i.d.\ Rademacher entries, Hunter, Kwan, Sauermann, and
  Sawhney in~\cite{HKSS} extended this bound to all $1\le k\le n$.
  Polavarapu in~\cite{Polavarapu} subsequently obtained the same bound
  for independent, not necessarily identically distributed, real entries
  whose atoms are uniformly bounded away from one. See also~\cite{DaiSongWang,DaiSuWang}
  for related results on the rank and
  singular values of subgaussian matrices.

  Sharper asymptotics are known in more restricted ranges of $k$. For i.i.d.\
  Rademacher matrices, Tikhomirov in~\cite{Tikhomirov} determined the sharp exponential singularity rate $(1/2+o(1))^n$, and Hu, Song, and
  Wu in~\cite{HuSongWu2026} recently obtained the full asymptotic
  $(2+o(1))n^2 2^{-n}$. For i.i.d.\ $\Ber(p)$
  matrices with fixed $p\in(0,1/2]$, Jain, Sah, and Sawhney in~\cite{JSSRank} obtained
  the asymptotic $(1-p+o(1))^{kn}$ for fixed $k$, and
  Song and Wang in~\cite{SongWang} extended it to $k=o(\sqrt{\log n})$.

  For symmetric matrices, Han in~\cite{Han} proved an $\exp(-cnk)$ bound for
  $1\le k\le c'\sqrt n$ when the entries on and above the diagonal of
  the matrix are i.i.d.\ copies of a centered, unit variance
  subgaussian random variable, where positive constants $c, c'$ depend only on the entry distribution.

  None of these results applies directly to a uniform matrix in
  $\Mnd$. The simultaneous row and column sum
  constraints create strong dependence among the entries.
  The singularity problem for uniform directed regular matrices has
  been studied in several degree regimes. 
  Cook in~\cite[Theorem~1.2]{Cook} proved that there are absolute constants $C,c>0$ such that, whenever $ \min(d,n-d)\ge C\log^2 n$
  one has
  \[
   \Prob\bigl[\rank(A)<n\bigr]=O(d^{-c}).
  \]
  Litvak, Lytova, Tikhomirov,
  Tomczak-Jaegermann, and Youssef in~\cite{litvak2019smallest} later obtained a quantitative lower
  bound for the smallest singular value of the shifted matrix. There are
  absolute constants $C,c>0$ such that, for every fixed $z\in\mathbb C$ with $|z|\le d/6$
  \[
   \Prob\bigl[s_n(A-zI)< n^{-6}\bigr]
   \le C\frac{\log^2 d}{\sqrt d}, \qquad  C<d<\frac{cn}{(\log n)(\log\log n)}.
  \]
  The same authors in~\cite{LLTTY} proved that there is a positive constant $C$ such that for every integer $d\ge C$ and every
  $n>d$,
  \[
   \Prob\bigl[\rank(A)\le n-2\bigr]
   \le C\frac{\log^2 d}{\log n}.
  \]
  Huang in~\cite{Huang} proved a polynomial upper bound for the
  singularity probability for every fixed $d\ge3$. Nguyen and Pan in~\cite{NguyenPan} later improved
  this bound to
  \[
   \Prob\bigl[\rank(A)<n\bigr]
   \le n^{-1/3+\varepsilon}
  \]
  for every fixed $d\ge3$ and $\varepsilon>0$, and all sufficiently large $n$.

  In the dense regime, Jain, Sah, and Sawhney in~\cite{JSS} proved
  that, for fixed $0<\lambda\le1/2$ and $\min(d,n-d)\ge\lambda n$,
  \[
   \Prob\bigl[s_n(A)\le\kappa\bigr]
   \le C_\lambda\kappa\sqrt n+2e^{-c_\lambda n}\, ,
   \qquad \kappa\ge0,
  \]
  where $C_\lambda,c_\lambda>0$ depend only on
  $\lambda$. Taking $\kappa=0$ gives
  \[
   \Prob\bigl[\rank(A)<n\bigr]
   \le 2e^{-c_\lambda n}.
  \]
  This confirms the conjecture of Cook in~\cite[Conjecture~1]{Cook} that the singularity probability is
  exponentially small in the dense regime.

  A closely related problem is obtained by placing independent signs
  on the edges. Let
  \[
   A^\pm=(A_{ij}\xi_{ij})_{1\le i,j\le n},
  \]
  where the $\xi_{ij}$ are i.i.d.\ Rademacher random variables
  independent of $A$. In this signed model, H. Liang, Y. Shi, K.~Yu, and
  T.~Yu in~\cite{Signed} proved
  \[
   \Prob\bigl[\rank(A^\pm)\le n-k\bigr]
   \le e^{-c_\lambda nk}\, ,
   \qquad 1\le k\le n,
  \]
  under the same degree regime. Conditional
  on $A$, the nonzero entries of $A^\pm$ remain independent random
  signs. The unsigned matrix $A$ has no corresponding entrywise
  randomness after its support is fixed.

  Theorem~\ref{thm:main} therefore supplies the large corank analogue
  of the Jain--Sah--Sawhney singularity bound for the unsigned
  model.

  \subsection{Overview of the proof}

  Let $M$ have independent $\Ber(d/n)$ entries. Conditioning on
  $M\in\Mnd$ gives the regular model at a cost of
  $e^{O_\lambda(n\log n)}$. The bound of Polavarapu in~\cite{Polavarapu}
  therefore applies when $k$ is at least a sufficiently large multiple
  of $\log n$; the singularity bound treats bounded $k$. For the
  remaining range, we split according to whether $\ker(A)$ contains a
  nonzero vector supported on at most $\delta n$ coordinates.

  \medskip\noindent\emph{Sparse kernel vectors.}
  We first treat the left kernel; the right kernel bound follows by
  transposition. Following the argument of Cook in~\cite{Cook}, we
  condition on all but
  two rows. The column sums fix their sum, while a sparse relation
  forces their difference into a subspace of small dimension. Replacing
  the two rows by their sum changes the rank by at most one. This gives
  the relative bound
  $e^{-cn}\Prob[\corank(A)\ge k-1]$ in
  Lemma~\ref{lem:sparse-contraction}.

  \medskip\noindent\emph{Pairing the rows.}
  Jain, Sah, and Sawhney in~\cite[proof of Lemma~2.18]{JSS} use a
  random row pairing to study a fixed vector. We choose a uniform
  random pairing independently of $A$ and reveal the pair sums. Conditional on this information, the row
  differences are independent. The pair sums and selected differences
  determine a possible row space, and every remaining difference must
  lie in it. For corank $h$, the left kernel gives
  $\lceil h/2\rceil$ disjoint pairs whose inclusion in the random
  pairing leaves at least that many differences not selected for a basis.
  This occurs with probability at least $n^{-\lceil h/2\rceil}$, uniformly
  in $A$; see Proposition~\ref{prop:matching-reduction}.

  \medskip\noindent\emph{Exceptional indices.}
  An exceptional probability of $e^{-cn}$ is too large when $h$ grows.
  We allow a small multiple of $h$ indices with small support or small
  codimension in the space of vectors on that support whose coordinates
  sum to zero. If this codimension is small at too many indices, the
  absence of sparse vectors in the right kernel forces disjoint sets of
  indices whose support sets leave a large complement. Section~\ref{sec:supports}
  bounds the probability of $r$ such sets by $e^{-cnr}$, leaving enough
  differences for the corank bound.

  \medskip\noindent\emph{Relative anticoncentration.}
  The inequality of Hunter, Kwan, Sauermann, and Sawhney in~\cite{HKSS}
  gives a comparison between the probabilities that a Rademacher vector
  and a lazy vector lie in a subspace. We apply it on a
  fixed support and condition both vectors to have sum zero. To pass
  from factors of the form $e^{-ch}$ to $e^{-cnh}$, we use their thin and
  thick subspace argument. For thick spaces, we sample a small multiple
  of $n$ independent auxiliary vectors from a mixture of the comparison
  laws, as in the argument of Polavarapu in~\cite{Polavarapu}. We first choose some
  of them to extend a basis of the pair sums to a basis of their joint
  span, and then extend it with row differences. The remaining
  auxiliary vectors must lie in the smaller
  span. Lemma~\ref{lem:transversality} shows that the pair sums have
  codimension of order $n$ on every large support, except with
  probability $e^{-cn^2}$; this remains true after adding the selected
  auxiliary vectors. Together with the relative inequality, this gives
  the required bound.

  \subsection{Notation}\label{sec:notation}

  Write $[N]=\{1,\ldots,N\}$ and
  $\supp(v)=\{j:v_j\ne0\}$ for $v\in\R^N$.
  A vector is $s$-sparse if $|\supp(v)|\le s$. For $0<\delta<1$, set
  \[
   \mathrm{Sparse}(\delta)
   :=\{v\in\R^n:|\supp(v)|\le\delta n\}.
  \]
  For $T\subseteq[N]$, write $T^c=[N]\setminus T$, let $v|_T$ be the
  restriction of $v$ to $T$, and write $S\mathbin\triangle T$ for the
  symmetric difference of sets. Let $\one$ be the vector of the
  appropriate dimension whose coordinates all equal $1$, and let
  $\one_T$ be the corresponding vector in $\R^T$. We use
  $\langle\cdot,\cdot\rangle$ for the standard inner product and identify
  row and column vectors with the same coordinates. For a matrix $B$, we write $B_i$ for
  the $i$-th row of $B$.

  For an $n\times n$ matrix $B$, set
  $\corank(B)=n-\rank(B)$. We write $\ker(B)$ for its right kernel and
  $\ker(B^{\mathsf T})$ for its left kernel. For a subspace
  $V\subseteq\R^N$, write $V^\perp$ for its orthogonal complement and
  $\codim(V)=N-\dim(V)$. If $W\subseteq V$, write
  $\codim_V(W)=\dim(V)-\dim(W)$. A subspace is rational if it has a
  basis with rational coordinates. We write $\Ber(p)$ and $\Bin(N,p)$
  for Bernoulli and binomial laws.

  Throughout, fix $0<\lambda\le1/2$ and
  $\lambda n\le d\le(1-\lambda)n$. Set
  $\lambda_0=d/n$ and $\theta=\lambda(1-\lambda)$, so that
  \begin{equation}\label{eq:density-range}
   \lambda_0\in[\lambda,1-\lambda],\qquad
   \lambda_0(1-\lambda_0)\ge\theta.
  \end{equation}
  Unless stated otherwise, $A$ is uniform on $\Mnd$. Constants
  $c,C>0$ may change from line to line and depend only on $\lambda$.
  We write $f=O_\alpha(g)$ when $|f|\le C_\alpha g$, and $f\asymp g$
  when $cg\le f\le Cg$ for absolute positive constants. 

  \begingroup
  \subsection{Organization}

  Section~\ref{sec:reduction} reduces Theorem~\ref{thm:main} to a
  relative bound for sparse kernel vectors and a tail bound when the
  right kernel contains no nonzero sparse vector.
  Section~\ref{sec:preliminaries} gives the comparison
  with the independent model and proves the sparse kernel bound.
  Section~\ref{sec:pairing} gives the conditional law of the row
  differences and proves the matching bound. Section~\ref{sec:supports}
  proves the support and codimension bounds. Finally,
  Section~\ref{sec:conditional} uses relative anticoncentration and
  the thin and thick subspace argument to prove
  Proposition~\ref{cor:nonsparse-tail}.
  \endgroup

  \section{Proof of Theorem~\ref{thm:main}}\label{sec:reduction}

  Write $P_k=\Prob[\corank(A)\ge k]$ for $0\le k\le n$, so $P_0=1$.
  We split according to whether the right kernel contains a nonzero
  sparse vector.

  \begin{lemma}\label{lem:sparse-contraction}
  There exist $\delta,c_{\ref{lem:sparse-contraction}}>0$, depending
  only on $\lambda$, such that, for all sufficiently large $n$,
  $1\le k\le n$, and $M\in\{A,A^{\mathsf T}\}$,
  \begin{equation}\label{eq:sparse-contraction}
   \Prob[\corank(M)\ge k,\
         \ker(M)\cap\mathrm{Sparse}(\delta)\ne\{0\}]
   \le e^{-c_{\ref{lem:sparse-contraction}}n}P_{k-1}.
  \end{equation}
  \end{lemma}

  Fix $\delta$ as in Lemma~\ref{lem:sparse-contraction}. The other case
  is covered by the following proposition.

  \begin{proposition}\label{cor:nonsparse-tail}
  There exist $c_{\ref{cor:nonsparse-tail}}>0$ and an integer
  $k_{\ref{cor:nonsparse-tail}}\ge1$, depending only on $\lambda$,
  such that, for all sufficiently large $n$ and
  $k_{\ref{cor:nonsparse-tail}}\le k\le n$,
  \[
   \Prob[\corank(A)\ge k,\
         \ker(A)\cap\mathrm{Sparse}(\delta)=\{0\}]
   \le e^{-c_{\ref{cor:nonsparse-tail}}nk}.
  \]
  \end{proposition}

  The bound for large $k$ follows from Lemma~\ref{lem:global-conditioning};
  Section~\ref{sec:nonsparse} treats the remaining range.

  \begin{proof}[Proof of Theorem~\ref{thm:main}]
  Jain, Sah, and Sawhney in~\cite[Theorem~1.1]{JSS} give
  $P_1\le2e^{-c'n}$ for some $c'>0$ depending only on $\lambda$.
  Lemma~\ref{lem:sparse-contraction} and
  Proposition~\ref{cor:nonsparse-tail} give
  \begin{equation}\label{eq:tail-recursion}
   P_k\le e^{-c_{\ref{cor:nonsparse-tail}}nk}
         +e^{-c_{\ref{lem:sparse-contraction}}n}P_{k-1}
  \end{equation}
  for all sufficiently large $n$ and
  $k\ge k_{\ref{cor:nonsparse-tail}}$.
  Choose an integer $K\ge k_{\ref{cor:nonsparse-tail}}$ so that
  \eqref{eq:tail-recursion} applies for $n>K$ and
  $2e^{-c_{\ref{lem:sparse-contraction}}K/2}\le1$. Take
  \[
   c_\lambda=\frac12\min\left\{
     c_{\ref{lem:sparse-contraction}},
     c_{\ref{cor:nonsparse-tail}},\frac{c'}K
   \right\}.
  \]
  For $k\le\min(K,n)$, $P_k\le P_1\le2e^{-c_\lambda nk}$; this also
  covers $n\le K$. For $K<k\le n$, the induction hypothesis and
  \eqref{eq:tail-recursion} give
  \[
   P_k\le e^{-c_\lambda nk}
   \left(e^{-c_{\ref{cor:nonsparse-tail}}nk/2}
         +2e^{-c_{\ref{lem:sparse-contraction}}n/2}\right)
   \le2e^{-c_\lambda nk},
  \]
  since $n>K$. This completes the induction.
  \end{proof}

  \section{Conditioning and sparse kernel vectors}\label{sec:preliminaries}

  \subsection{Conditioning on regularity}

  Let $\widetilde A$ have independent $\Ber(\lambda_0)$ entries.
  Conditioning on $\widetilde A\in\Mnd$ gives the regular model. We
  use this fact for large $k$ and a local comparison for events
  determined by few rows.

  \begin{lemma}\label{lem:global-conditioning}
  There exist
  $C_{\ref{lem:global-conditioning}},
   c_{\ref{lem:global-conditioning}},
   k_{\ref{lem:global-conditioning}}>0$,
  depending only on $\lambda$, such that, for all sufficiently large $n$,
  \[
   \Prob[\widetilde A\in\Mnd]
   \ge e^{-C_{\ref{lem:global-conditioning}}n\log n}.
  \]
  Conditional on this event, $\widetilde A$ is uniform on $\Mnd$.
  Moreover, for $k_{\ref{lem:global-conditioning}}\log n\le k\le n$,
  \begin{equation}\label{eq:large-corank-conditioning}
   \Prob[\corank(A)\ge k]
   \le e^{-c_{\ref{lem:global-conditioning}}nk}.
  \end{equation}
  \end{lemma}

  \begin{proof}
  Let $Z$ be the matrix of Liang, Shi, Yu, and Yu
  in~\cite[Lemma~4.2]{Signed}. Since the matrix $(|Z_{ij}|)_{i,j=1}^n$ has the law of
  $\widetilde A$, that lemma gives the first bound. Every matrix in
  $\Mnd$ has $nd$ ones and hence the same probability under
  $\widetilde A$, proving the assertion about the conditional law.
  As $\max\{\lambda_0,1-\lambda_0\}\le1-\lambda$, the result of
  Polavarapu in~\cite[Theorem~1.1]{Polavarapu} gives
  $\Prob[\corank(\widetilde A)\ge k]\le e^{-c_0nk}$, where $c_0>0$
  depends only on $\lambda$. Thus
  \[
   \Prob[\corank(A)\ge k]
   \le e^{-c_0nk+C_{\ref{lem:global-conditioning}}n\log n}.
  \]
  Increasing $k_{\ref{lem:global-conditioning}}$ and taking
  $c_{\ref{lem:global-conditioning}}\le c_0/2$ proves the result.
  \end{proof}

  \begin{proposition}\label{prop:local-comparison}
  Fix $L>0$. There exists $C_{\ref{prop:local-comparison}}>0$,
  depending only on $\lambda,L$, such that, for all sufficiently large
  $n$, every integer $1\le q\le L\log n$, and every event $\mathcal E$
  determined by $q$ specified rows,
  \[
   \Prob[A\in\mathcal E]
   \le e^{C_{\ref{prop:local-comparison}}(q^2+q\log n)}
         \Prob[\widetilde A\in\mathcal E].
  \]
  \end{proposition}

  \begin{proof}
  This follows from Lemma~2.1 of Liang, Shi, Yu, and Yu
  in~\cite{Signed}, with $m=q$ and $K=L$.
  \end{proof}

  \subsection{Sparse vectors in the kernel}\label{sec:sparse}

  We first prove the bound in Lemma~\ref{lem:sparse-contraction}
  for the left kernel by conditioning on all but two rows, following
  the argument of Cook in~\cite[Section~3.1]{Cook}. The right kernel
  bound follows by transposition.

  \begin{proof}[Proof of Lemma~\ref{lem:sparse-contraction}]
  Fix $i\in[n]$. Counting each column according to its entry in row
  $i$ gives
  \begin{equation}\label{eq:symmetric-difference-sum}
   \sum_{j=1}^n
   |\supp(A_i)\mathbin\triangle\supp(A_j)|
   =2d(n-d)\ge2\theta n^2.
  \end{equation}
  Each summand is at most $n$, so at least $\theta n$ indices $j$
  satisfy $|\supp(A_i)\mathbin\triangle\supp(A_j)|\ge\theta n$;
  otherwise the sum in~\eqref{eq:symmetric-difference-sum} is less
  than $2\theta n^2$.

  Fix $S\subseteq[n]$, $i\in S$, and $j\notin S$. Reveal all rows
  except $A_i,A_j$, and let $\mathcal F$ be the resulting
  sigma-algebra. The column sums determine $A_i+A_j$, and hence
  $T=\supp(A_i)\mathbin\triangle\supp(A_j)$ and $N=|T|$. Set
  $D=A_i-A_j$. Conditional on $\mathcal F$, $D|_T$ is uniform on
  \[
   \{x\in\{\pm1\}^T:\langle x,\one_T\rangle=0\}.
  \]
  Indeed, the two rows are equal on $T^c$ and have equal row sums, so
  each has $N/2$ ones on $T$. Each choice of these ones for $A_i$
  determines one matrix, and all such matrices have the same
  conditional probability.

  If $u^{\mathsf T}A=0$, $\supp(u)\subseteq S$, and $u_i\ne0$, then
  on $T$,
  \[
   D|_T=-\one_T-\frac2{u_i}
         \sum_{\ell\in S\setminus\{i\}}u_\ell A_\ell|_T
   \in W:=\operatorname{span}\bigl(\one_T,
         \{A_\ell|_T:\ell\in S\setminus\{i\}\}\bigr).
  \]
  Since $j\notin S$, the space $W$ is $\mathcal F$-measurable and
  $\dim W\le|S|$; it does not depend on $u$.
  An $r$-dimensional subspace of $\R^N$ contains at most $2^r$ sign
  vectors, since projection onto some $r$ coordinates is injective.
  Also $\binom N{N/2}\ge2^N/(N+1)$. Therefore, for $N\ge\theta n$,
  \[
   \Prob[D|_T\in W\mid\mathcal F]
   \le(N+1)2^{|S|-N}.
  \]
  Choose $0<\delta<\theta/4$ and let $c_0=\theta\log2/2$.
  For $|S|\le\delta n$ and all sufficiently large $n$, this is
  at most $e^{-c_0n}$.

  Let $A'$ be the $(n-1)\times n$ matrix obtained by replacing
  $A_i,A_j$ with their sum. It is
  $\mathcal F$-measurable and
  $\rank(A')\le\rank(A)\le\rank(A')+1$. Hence
  $\corank(A)\ge k$ implies $\rank(A')\le n-k$, while
  $\rank(A')\le n-k$ implies $\corank(A)\ge k-1$ for every matrix
  with the revealed rows. The sparse relation also gives $D|_T\in W$.
  Averaging the conditional bound over $\mathcal F$ gives, for each
  $S,i,j$,
  \[
  \begin{aligned}
   &\Prob\bigl[\corank(A)\ge k,\ N\ge\theta n,\
         \exists u:\,u^{\mathsf T}A=0,\
         \supp(u)\subseteq S,\ u_i\ne0\bigr]\\
   &\qquad\le e^{-c_0n}\Prob[\rank(A')\le n-k]
   \le e^{-c_0n}P_{k-1}.
  \end{aligned}
  \]

  Take $\delta$ smaller if necessary so that
  $-\delta\log\delta-(1-\delta)\log(1-\delta)<c_0/2$.
  The binomial sum bound gives
  $\sum_{s\le\delta n}\binom ns\le e^{c_0n/2}$ for all sufficiently
  large $n$. If $\ker(A^{\mathsf T})$ contains a nonzero vector in
  $\mathrm{Sparse}(\delta)$, take $S$ to be its support and $i\in S$.
  By~\eqref{eq:symmetric-difference-sum}, some $j\notin S$ has
  $N\ge\theta n$. A union bound over at most
  $n^2\sum_{s\le\delta n}\binom ns$ choices of $S,i,j$ gives
  \[
  \begin{aligned}
   &\Prob[\corank(A)\ge k,\
         \ker(A^{\mathsf T})\cap\mathrm{Sparse}(\delta)\ne\{0\}]\\
   &\qquad\le n^2e^{-c_0n/2}P_{k-1}
   \le e^{-c_0n/4}P_{k-1}.
  \end{aligned}
  \]
  Since $A$ and $A^{\mathsf T}$ have the same distribution, the bound
  also holds for the right kernel with
  $c_{\ref{lem:sparse-contraction}}=c_0/4$.
  \end{proof}

  \section{Pairing the rows}
  \label{sec:pairing}

  Let $A$ be uniform on $\Mnd$, and let $m:=\lfloor n/2\rfloor$.
  Independently of $A$, choose a uniform matching $\pi$ of $[n]$
  with $m$ pairs. Write its pairs as $\{a_i,b_i\}$, $i\in[m]$,
  where $a_i<b_i$. As in the proof of Lemma~2.18 by Jain,
  Sah, and Sawhney in~\cite{JSS}, we reveal the pair sums.
  Set
  \begin{equation}\label{eq:pair-sum-difference}
   S_i:=A_{a_i}+A_{b_i},\qquad
   D_i:=A_{a_i}-A_{b_i},\qquad
   T_i:=\{j\in[n]:S_i(j)=1\}.
  \end{equation}
  Thus $T_i=\supp(D_i)$. Let $U$ be the span of the pair sums
  and the remaining row, if $n$ is odd. This row is determined by
  the pair sums and the column sums, since
  $A_c=d\one^{\mathsf T}-\sum_{i=1}^m S_i$.
  In particular, $\dim U\le n-m$.

  Let $\mathcal G_\pi$ be generated by $\pi,S_1,\ldots,S_m$.
  Conditional on $\mathcal G_\pi$, the space $U$ and the sets $T_i$
  are fixed. We use the following conditional law.

  \begin{lemma}\label{lem:pair-product-law}
  Conditional on $\mathcal G_\pi$, the vectors $D_1,\ldots,D_m$
  are independent. Each $D_i$ is zero outside $T_i$, and
  $D_i|_{T_i}$ is uniform among the $\{\pm1\}$-vectors on $T_i$
  whose coordinates sum to zero.
  \end{lemma}

  \begin{proof}
  For fixed $\pi$ and pair sums,
  the choices of the differences form a product of the sets of sign
  vectors in the lemma. Each choice gives one matrix in $\Mnd$, so the
  uniform law of $A$ gives the conditional product law.
  \end{proof}

  Each row pair has the same span as its sum and difference. Hence
  \begin{equation}\label{eq:final-rowspace}
   V:=U+\operatorname{span}(D_1,\ldots,D_m)
     =\operatorname{span}(A_1,\ldots,A_n).
  \end{equation}
  For an integer $0\le h\le n$, let
  \begin{equation}\label{eq:quotient-deficit}
   \ell:=m-(n-h-\dim U).
  \end{equation}
  On $\{\corank(A)=h\}$, the differences span $V/U$, which has
  dimension $m-\ell$. Thus $m-\ell$ differences can be selected
  to extend a basis of $U$ to one of $V$, and each of the remaining
  $\ell$ differences lies in this span. Since $\dim U\le n-m$,
  $0\le\ell\le h$ on this event. We next show that an independent
  random pairing gives $\ell\ge\lceil h/2\rceil$ with sufficiently
  large probability.

  \begin{proposition}\label{prop:matching-reduction}
  Fix $A\in\Mnd$ with corank $h\ge1$. For a uniform maximum matching
  $\pi$ of $[n]$,
  \begin{equation}\label{eq:matching-reduction}
   \Prob_\pi\!\left[\ell\ge\left\lceil\frac h2\right\rceil\right]
   \ge n^{-\lceil h/2\rceil}.
  \end{equation}
  \end{proposition}

  \begin{proof}
  Let $L:=\ker(A^{\mathsf T})$. Since $A\one=d\one$ and $d>0$,
  $L\subseteq\one^\perp$ and $\dim L=h$. For a matching $F$,
  let $\Phi_F:L\to\R^{|F|}$ have coordinates $y_a-y_b$ for
  $\{a,b\}\in F$, with $a<b$.

  Choose $F$ with as many pairs as possible, subject to
  $\rank\Phi_F=|F|$. If $2|F|<h$, there is a nonzero $x\in L$
  that is zero on every vertex in $F$. Since $\sum_jx_j=0$, there are
  vertices $u,v$ outside $F$ with $x_u>0>x_v$. The map $y\mapsto y_u-y_v$
  does not belong to the span of the coordinates of $\Phi_F$,
  as its value at $x$ is nonzero. Adding $\{u,v\}$ contradicts the
  choice of $F$. Taking a subset if necessary, we may therefore
  assume $|F|=\lceil h/2\rceil$.

  For a maximum matching $\pi$, let
  $K_\pi:=\{y\in\R^n:y_{a_i}=y_{b_i}\text{ for all }i\in[m]\}$.
  Its dimension is $n-m$. The map $y\mapsto y^{\mathsf T}A$ from
  $K_\pi$ has image $U$: the common coordinates on a pair give
  the coefficient of $S_i$, and the remaining coordinate, if any,
  gives the coefficient of the remaining row. Its kernel is
  $L\cap K_\pi=\ker\Phi_\pi$. Hence
  \[
   \rank\Phi_\pi
   =h-(n-m-\dim U)=\ell.
  \]
  If $F\subseteq\pi$, then the coordinates of $\Phi_F$ are among
  those of $\Phi_\pi$, and $\ell\ge|F|$.

  For a uniform maximum matching on $N\ge2$ vertices, a given
  pair occurs with probability
  $\lfloor N/2\rfloor/\binom N2\ge1/N$.
  Conditional on containing prescribed disjoint pairs, the matching
  on the remaining vertices is uniform. Thus
  \[
   \Prob_\pi[F\subseteq\pi]
   \ge\prod_{j=0}^{|F|-1}(n-2j)^{-1}
   \ge n^{-|F|},
  \]
  which proves the proposition.
  \end{proof}

  Let $\mathcal E_h$ depend only on $A$ and imply
  $\corank(A)=h\ge1$. The matching remains uniform conditional
  on $A$, so Proposition~\ref{prop:matching-reduction}, averaged
  over $A$, gives
  \begin{equation}\label{eq:capture-inequality}
   \Prob_A[\mathcal E_h]
   \le n^{\lceil h/2\rceil}
   \Prob_{A,\pi}[\mathcal E_h,\ \ell\ge\lceil h/2\rceil].
  \end{equation}
  Here $\ell$ is fixed after conditioning on $\mathcal G_\pi$.
  On realizations with $\ell>m$, the event $\mathcal E_h$ is empty.
  We do not condition on $\mathcal E_h$ or the corank. For each
  possible set of $m-\ell$ selected differences, reveal only
  those differences. The others still have the independent laws in
  Lemma~\ref{lem:pair-product-law}. On $\mathcal E_h$, some such
  set extends a basis of $U$ to a basis of $V$, so we take a union
  bound over these sets and require every remaining difference
  to lie in the resulting span.

  \section{Support sets and codimension}
  \label{sec:supports}

  We use the notation of Section~\ref{sec:pairing}. The sets
  $T_i$ play the role of the switching sets in the work of
  Jain, Sah, and Sawhney in~\cite[Definition~2.16 (P1)]{JSS}.
  We count disjoint groups of indices for which the union of the
  corresponding sets $T_i$ has a large complement, following the
  argument of Liang, Shi, Yu, and Yu in~\cite[Lemma~2.5]{Signed}.
  Both counts below depend only on the pair sums.

  \subsection{The support sets}

  Recall $m=\lfloor n/2\rfloor$ and $\theta=\lambda(1-\lambda)$.
  Set $\eta:=\delta/2$, where $\delta$ is from
  Lemma~\ref{lem:sparse-contraction}, and let
  \begin{equation}\label{eq:small-support-count}
   N_\theta:=|\{i\in[m]:|T_i|<\theta n\}|.
  \end{equation}
  For a positive integer $Q$, let $\mathrm D(A,\pi)$ be the largest
  number of disjoint sets $G\subseteq[m]$, each of size
  $Q$, for which
  \begin{equation}\label{eq:defective-set}
   \left|[n]\setminus\bigcup_{i\in G}T_i\right|>\eta n.
  \end{equation}
  The sets of indices are disjoint; the sets $T_i$ need not be.
  This is Definition~2.4 of Liang, Shi, Yu, and Yu
  in~\cite{Signed}, with $T_i$ in place of their support sets.

  \begin{proposition}\label{prop:defect-tails}
  There are an integer $Q\ge1$ and $c_{\ref{prop:defect-tails}}>0$,
  depending only on $\lambda$, such that, for all sufficiently
  large $n$ and $1\le r\le k_{\ref{lem:global-conditioning}}\log n$,
  \[
   \Prob_{A,\pi}[N_\theta\ge r]
   +\Prob_{A,\pi}[\mathrm D(A,\pi)\ge r]
   \le2e^{-c_{\ref{prop:defect-tails}}nr}.
  \]
  \end{proposition}

  \begin{proof}
  Fix $\pi$. For the independent matrix $\widetilde A$, write
  $\widetilde T_i:=\{j:\widetilde A_{a_i j}\ne
  \widetilde A_{b_i j}\}$ and $\lambda_0:=d/n$.
  Then $|\widetilde T_i|$ has law
  $\Bin(n,2\lambda_0(1-\lambda_0))$ and mean at least
  $2\theta n$. Chernoff's inequality gives
  \[
   \Prob[|\widetilde T_i|<\theta n]\le e^{-\theta n/4}.
  \]
  For a fixed $Q$-set $G$, condition~\eqref{eq:defective-set}
  for $\widetilde T_i$ implies that on some set $J\subseteq[n]$
  of size $\lceil\eta n\rceil$ the two entries in each pair
  $i\in G$ are equal. These conditions involve disjoint entries.
  Thus a union bound over $J$ gives
  \[
   \Prob\!\left[\left|[n]\setminus\bigcup_{i\in G}
     \widetilde T_i\right|>\eta n\right]
   \le\binom n{\lceil\eta n\rceil}
      (\lambda_0^2+(1-\lambda_0)^2)^{Q\lceil\eta n\rceil}
   \le\left[\frac e\eta(1-2\theta)^Q\right]^{\lceil\eta n\rceil}.
  \]
  Choose $Q$ so that the last bound is at most $e^{-\eta n}$.
  For $r$ distinct indices with small $T_i$, or $r$ disjoint $Q$-sets
  satisfying~\eqref{eq:defective-set}, the corresponding events
  use disjoint row pairs. Their probability in the independent
  model is at most $e^{-cnr}$, where
  $c:=\min\{\theta/4,\eta\}$.
  Each event involves at most $2Qr$ specified rows. Apply
  Proposition~\ref{prop:local-comparison} with
  $L:=2Qk_{\ref{lem:global-conditioning}}+1$; it bounds the
  probability
  in the regular model by
  \[
   \exp\!\left\{-cnr+
   C_{\ref{prop:local-comparison}}(4Q^2r^2+2Qr\log n)
   \right\}.
  \]
  There are at most $m^{Qr}$ choices in either case.
  Since $r\le k_{\ref{lem:global-conditioning}}\log n$,
  the positive terms in the exponent, including $Qr\log m$,
  are $O(r\log n)$. For large $n$, each probability is at most
  $e^{-cnr/2}$. The bound is uniform in $\pi$; averaging over
  $\pi$ completes the proof.
  \end{proof}

  Fix $Q$ as in Proposition~\ref{prop:defect-tails}.

  \subsection{Codimension}

  For $i\in[m]$, let
  \begin{equation}\label{eq:Ei-Hi}
   E_i:=\{x\in\R^n:\supp(x)\subseteq T_i\},\qquad
   H_i:=\{x\in E_i:\langle x,\one\rangle=0\}.
  \end{equation}
  Thus $D_i\in H_i$. For a subspace $V\subseteq\R^n$, put
  $W:=V^\perp$ and
  \begin{equation}\label{eq:local-codimension}
   r_i(V):=\codim_{H_i}(V\cap H_i).
  \end{equation}
  The next lemma applies to every possible row space $V$;
  when $V$ is the row space of $A$, $W=\ker(A)$.
  We use the argument of Liang, Shi, Yu, and Yu
  in~\cite[Lemma~3.5]{Signed} for the sum-zero spaces $H_i$.

  \begin{lemma}\label{lem:visibility}
  Suppose $\dim W=h\ge4Q$ and $W$ contains no nonzero
  $\delta n$-sparse vector. Then
  \[
   \left|\left\{i\in[m]:r_i(V)<\frac{h}{2Q}\right\}\right|
   <Q(\mathrm D(A,\pi)+1).
  \]
  \end{lemma}

  \begin{proof}
  Fix a $Q$-set $G$ whose indices satisfy $r_i(V)<h/(2Q)$.
  Identify $\R^{T_i}$ with $E_i$ by extending vectors by zero,
  and write $W|_{T_i}:=\{w|_{T_i}:w\in W\}$.
  For $i\in G$, the orthogonal complement in $E_i$ of the
  restrictions $W|_{T_i}$ is $V\cap E_i$. Since $H_i$ has
  codimension at most one in $E_i$,
  \[
   \dim(W|_{T_i})
   =\codim_{E_i}(V\cap E_i)\le r_i(V)+1.
  \]
  It follows that the map
  $W\to\prod_{i\in G}\R^{T_i}$ taking $w$ to
  $(w|_{T_i})_{i\in G}$ has rank less than
  $h/2+Q\le3h/4<h$. A nonzero vector in its kernel is
  supported on $[n]\setminus\bigcup_{i\in G}T_i$.
  Its support has size greater than $\delta n>\eta n$, so
  $G$ satisfies~\eqref{eq:defective-set}.
  If there were at least $Q(\mathrm D(A,\pi)+1)$ such indices,
  we could split them into $\mathrm D(A,\pi)+1$ disjoint
  $Q$-sets satisfying~\eqref{eq:defective-set}, a contradiction.
  \end{proof}

  For the auxiliary vectors in Section~\ref{sec:conditional},
  we need a bound of order $n$ for $r_i(U)$. Each added vector
  reduces $r_i$ by at most one, so the following lemma suffices.

  \begin{lemma}\label{lem:transversality}
  There are $\beta,c_{\ref{lem:transversality}}>0$,
  depending only on $\lambda$, such that, for all sufficiently
  large $n$,
  \[
   \Prob_{A,\pi}\left[\exists i\in[m]:
   |T_i|\ge\theta n,\ r_i(U)<\beta n\right]
   \le e^{-c_{\ref{lem:transversality}}n^2}.
  \]
  \end{lemma}

  \begin{proof}
  Let $\beta:=\theta/8$ and fix $\pi$ and $i\in[m]$.
  The next calculation holds for every $0$--$1$ matrix.
  Let $S$ have rows $S_1,\ldots,S_m$, let $U_0$ be their span,
  and write $T:=T_i$ and $t:=|T|$. The space $U$ is obtained
  from $U_0$ by adding at most one row. If $t\ge\theta n$ and
  $r_i(U)<\beta n$, then $\dim H_i=t-1$ gives
  \[
   \dim(U_0\cap E_i)
   \ge\dim(U\cap H_i)-1>t-\beta n-2.
  \]
  Restriction of $U_0$ to $T^c$ has image the row space of
  $S_{[m],T^c}$ and kernel $U_0\cap E_i$. Thus
  \[
   \rank(S_{[m],T^c})
   =\dim U_0-\dim(U_0\cap E_i)
   <m-t+\beta n+2.
  \]
  Deleting row $i$ cannot increase rank. Moreover,
  \[
   \min(m-1,n-t)-(m-t+\beta n+2)
   =\min(t-1,n-m)-\beta n-2
   \ge\theta n/4
  \]
  for all sufficiently large $n$, since $n-m\ge n/2$
  and $\theta\le1/4$. Hence
  \begin{equation}\label{eq:rank-deficit-submatrix}
   \rank(S_{[m]\setminus\{i\},T^c})
   \le\min(m-1,n-t)-\left\lceil\frac{\theta n}{4}\right\rceil.
  \end{equation}

  For a fixed $\pi$, let $\mathcal B_i$ be the event on the
  left side of the lemma for this $i$, defined for either the
  regular or the independent matrix. Conditional on the two
  rows $\widetilde A_{a_i},\widetilde A_{b_i}$ of the independent
  matrix, $T$ is fixed. The entries of
  $\widetilde S_{[m]\setminus\{i\},T^c}$ are independent
  $\Bin(2,\lambda_0)$ variables, $\lambda_0=d/n$, with largest
  atom at most $1-2\theta<1$. If
  $\min(m-1,n-t)<\lceil\theta n/4\rceil$, the event
  $\mathcal B_i$ is empty by~\eqref{eq:rank-deficit-submatrix}.
  Otherwise, the result of Polavarapu
  in~\cite[Corollary~1.2]{Polavarapu} gives, uniformly over the
  revealed pair,
  \[
   \Prob[\widetilde A\in\mathcal B_i
         \mid\widetilde A_{a_i},\widetilde A_{b_i}]
   \le\exp\!\left\{-c\min(m-1,n-t)
          \left\lceil\frac{\theta n}{4}\right\rceil\right\}
   \le e^{-cn^2}.
  \]
  Averaging over the pair and taking a union bound over $i$
  gives $\Prob[\widetilde A\in\bigcup_i\mathcal B_i]
  \le ne^{-cn^2}$.

  Every matrix in $\Mnd$ has the same probability under the
  independent $\Ber(\lambda_0)$ law. Conditioning
  $\widetilde A$ to belong to $\Mnd$ therefore gives the law
  of $A$. By Lemma~\ref{lem:global-conditioning},
  \[
   \Prob_A\left[\bigcup_i\mathcal B_i\right]
   \le\frac{ne^{-cn^2}}{\Prob[\widetilde A\in\Mnd]}
   \le\exp\{-cn^2+C_{\ref{lem:global-conditioning}}n\log n
              +\log n\}
   \le e^{-c_{\ref{lem:transversality}}n^2}.
  \]
  This holds for every fixed $\pi$; we average over $\pi$.
  \end{proof}

  \section{Relative anticoncentration and the conditional bound}
  \label{sec:conditional}

  We adapt the relative bound and the thin and thick subspace
  argument of Hunter, Kwan, Sauermann, and Sawhney
  in~\cite[Proposition~2.1 and Claims~2.2--2.3]{HKSS} to the
  conditional law of Lemma~\ref{lem:pair-product-law}. Each row
  difference has coordinate sum zero, and the support sets may differ.

  \subsection{Relative anticoncentration}

  For $0<p\le1/2$, let $\mu_p(1)=\mu_p(-1)=p$ and
  $\mu_p(0)=1-2p$. Fix $p=1/32$. For even $N\ge2$, put
  $H_0=\{x\in\R^N:\langle x,\one\rangle=0\}$.
  Let $X$ and $Y$ have laws $\mu_{1/2}^{\otimes N}$ and
  $\mu_p^{\otimes N}$, respectively, conditioned to lie in $H_0$.

  \begin{lemma}\label{lem:balanced-relative}
  There are absolute constants
  $k_{\ref{lem:balanced-relative}}\in\mathbb N$,
  $C_{\ref{lem:balanced-relative}}>0$, and $\gamma\in(0,1)$
  such that, for every rational subspace $V\subseteq\R^N$, if
  $r=\codim_{H_0}(V\cap H_0)$, then
  \begin{equation}\label{eq:balanced-relative-bound}
   \Prob[X\in V]\le\gamma^r\Prob[Y\in V]
   \qquad\text{for }r\ge k_{\ref{lem:balanced-relative}},
  \end{equation}
  and, for all $r\ge0$,
  \begin{equation}\label{eq:balanced-lazy-absolute}
   \Prob[Y\in V]\le
   C_{\ref{lem:balanced-relative}}\sqrt N\,(1-2p)^{r+1}.
  \end{equation}
  \end{lemma}

  \begin{proof}
  Let $W=V\cap H_0$, and let $X'\sim\mu_{1/2}^{\otimes N}$ and
  $Y'\sim\mu_p^{\otimes N}$.
  The rational subspace $W$ has codimension $r+1$ in $\R^N$,
  and $W\cap\Q^N$ has the same codimension in $\Q^N$.
  The inequality of Hunter, Kwan, Sauermann, and Sawhney
  in~\cite[Proposition~2.1]{HKSS} gives, for some $\gamma_0<1$
  and all sufficiently large $r$,
  \[
   \Prob[X'\in W]\le\gamma_0^{r+1}\Prob[Y'\in W].
  \]
  Moreover,
  \begin{align*}
   \Prob[X'\in H_0]
   &=2^{-N}\binom N{N/2}\asymp N^{-1/2},\\
   \Prob[Y'\in H_0]
   &=\frac1{2\pi}\int_{-\pi}^{\pi}
      (1-2p+2p\cos t)^N\,dt\asymp N^{-1/2}.
  \end{align*}
  The second comparison follows from $1-\cos t\asymp t^2$ on
  $[-\pi,\pi]$. Since $W\subseteq H_0$, conditioning on $H_0$
  therefore gives
  \[
   \Prob[X\in V]\le C\gamma_0^{r+1}\Prob[Y\in V].
  \]
  Taking $\gamma=\sqrt{\gamma_0}$ and increasing
  $k_{\ref{lem:balanced-relative}}$ proves
  \eqref{eq:balanced-relative-bound}.

  For the second bound, $W$ is given by $r+1$ independent linear
  equations. Choose $r+1$ coordinates determined by these equations
  once the other coordinates are fixed. Each atom of $\mu_p$ is at
  most $1-2p$, so
  $\Prob[Y'\in W]\le(1-2p)^{r+1}$.
  Divide by $\Prob[Y'\in H_0]\asymp N^{-1/2}$.
  \end{proof}

  \subsection{The conditional bound}

  Fix a realization of $\mathcal G_\pi$ of positive probability
  outside the event in Lemma~\ref{lem:transversality}. Thus
  \[
   r_i(U)\ge\beta n\qquad\text{whenever }|T_i|\ge\theta n.
  \]
  In this subsection probabilities are conditional on this realization.
  The space $U$ and sets $T_i$ are fixed; once $h$ is fixed, so is
  $\ell$. Put $X_i=D_i$ and
  $V_X=U+\operatorname{span}(X_1,\ldots,X_m)$.
  By Lemma~\ref{lem:pair-product-law}, the $X_i$ are independent.
  For each $i$, let $Z_i$ be zero off $T_i$ and have law
  $\mu_p^{\otimes T_i}$ conditioned on coordinate sum zero on $T_i$;
  put $Z_i=0$ if $T_i=\varnothing$. Choose the $Z_i$ independently
  of each other and of the $X_i$.

  Each nonempty $T_i$ has even size by
  Lemma~\ref{lem:pair-product-law}. The restriction of
  $V\cap H_i$ to $T_i$ has codimension $r_i(V)$ in the sum-zero
  hyperplane of $\R^{T_i}$. Applying
  Lemma~\ref{lem:balanced-relative} in $H_i$ gives, for rational $V$,
  \begin{equation}\label{eq:relative-consequence}
   \Prob[X_i\in V]
   \le\gamma^{h/(2Q)}\Prob[Z_i\in V]
   \qquad\text{if }r_i(V)\ge h/(2Q)
  \end{equation}
  provided $h\ge2Qk_{\ref{lem:balanced-relative}}$.
  The spaces below are rational, since they are spanned by vectors
  with integer coordinates. Also, if a rational subspace
  $V'\supseteq U$ satisfies
  $\dim(V'/U)\le\beta n/2$, then
  $r_i(V')\ge r_i(U)-\dim(V'/U)\ge\beta n/2$ for
  $|T_i|\ge\theta n$. Hence, for all sufficiently large $n$,
  \begin{equation}\label{eq:absolute-consequence}
   \Prob[Z_i\in V']
   \le C_{\ref{lem:balanced-relative}}\sqrt n\,
         (1-2p)^{\beta n/2+1}
   \le(1-p)^{\beta n/2}
   \qquad(|T_i|\ge\theta n).
  \end{equation}
  The last inequality uses $1-2p<1-p$.

  \begin{proposition}\label{prop:fixed-skeleton}
  There are constants $c_{\ref{prop:fixed-skeleton}}>0$ and
  $k_{\ref{prop:fixed-skeleton}}\in\mathbb N$, depending only on
  $\lambda$, such that the following holds for all sufficiently
  large $n$. Let $h\ge k_{\ref{prop:fixed-skeleton}}$ be an integer
  and let $\ell$ be as in~\eqref{eq:quotient-deficit}. If
  \begin{equation}\label{eq:fixed-skeleton-input}
   \ell-N_\theta-Q\bigl(\mathrm D(A,\pi)+1\bigr)\ge h/4,
  \end{equation}
  then
  \[
   \Prob[\codim(V_X)=h,\ V_X^\perp
       \cap\mathrm{Sparse}(\delta)=\{0\}]
   \le e^{-c_{\ref{prop:fixed-skeleton}}nh}.
  \]
  \end{proposition}

  \begin{proof}
  Take $k_{\ref{prop:fixed-skeleton}}$ sufficiently large that
  $h\ge\max\{2Qk_{\ref{lem:balanced-relative}},4Q\}$, and let
  \[
   q_*:=N_\theta+Q\bigl(\mathrm D(A,\pi)+1\bigr).
  \]
  The event is empty unless $0\le\ell\le m$. Consider the possible
  values $V$ of $V_X$ on the event in the proposition. Each contains
  $U$, is rational, has codimension $h$, and satisfies
  $V^\perp\cap\mathrm{Sparse}(\delta)=\{0\}$. For each such $V$, put
  \[
   E(V):=\{i:|T_i|<\theta n\}
         \cup\{i:r_i(V)<h/(2Q)\}.
  \]
  Lemma~\ref{lem:visibility} gives $|E(V)|<q_*$. As $q_*$ is an
  integer and $m\ge2$, there are at most
  $\sum_{j<q_*}\binom mj\le m^{q_*}$ possible sets $E(V)$.

  Partition these spaces according to $E(V)$.
  Fix one class, write $E$ for its common set $E(V)$, and let
  $\mathcal V$ be the corresponding set of spaces. We bound
  $\Prob[V_X\in\mathcal V]$ before summing over classes.
  We have not conditioned on $E(V_X)=E$, so the $X_i$ remain
  independent. Write
  \[
   N=m-|E|,\qquad s=\ell-|E|\ge h/4.
  \]
  Here $N\ge s>0$. For $V\in\mathcal V$ and $i\notin E$,
  $|T_i|\ge\theta n$ and $r_i(V)\ge h/(2Q)$.
  Also, $\dim(V/U)=n-h-\dim U=m-\ell$.
  Set
  \[
   \rho_V=\frac1N\sum_{i\in[m]\setminus E}\Prob[X_i\in V].
  \]
  The $X_i$ need not have the same law, so we use this average.
  Maclaurin's inequality, as used by Polavarapu
  in~\cite[Claim~4.4]{Polavarapu},
  gives, for every $S\subseteq[m]\setminus E$ of size $t$,
  \begin{equation}\label{eq:maclaurin-product-bound}
   \prod_{i\in S}\Prob[X_i\in V]
   \le\sum_{\substack{S'\subseteq[m]\setminus E\\|S'|=t}}
         \prod_{i\in S'}\Prob[X_i\in V]
   \le\binom Nt\rho_V^t\le2^N\rho_V^t.
  \end{equation}
  Call $V\in\mathcal V$ thin if
  $\rho_V\le(1-p)^{\beta n/2}$, and thick otherwise; denote
  the two classes by $\mathcal V_{\mathrm{thin}}$ and
  $\mathcal V_{\mathrm{thick}}$.

  \medskip\noindent\emph{Thin subspaces.}
  If $V_X\in\mathcal V_{\mathrm{thin}}$, some set
  $I\subseteq[m]$ of size $m-\ell$ satisfies
  $V_I:=U+\operatorname{span}(X_i:i\in I)=V_X$.
  Fix $I$ and reveal $(X_i)_{i\in I}$. The remaining vectors are
  independent, and on $V_I\in\mathcal V_{\mathrm{thin}}$ at least
  $m-|I|-|E|=s$ indices outside $I$ also lie outside $E$.
  By~\eqref{eq:maclaurin-product-bound}, the conditional probability
  that all remaining $X_j$ lie in $V_I$ is at most
  $2^N\rho_{V_I}^s\le2^N(1-p)^{\beta ns/2}$.
  Averaging over the selected vectors and summing over $I$ gives
  \begin{equation}\label{eq:thin-bound}
   \Prob[V_X\in\mathcal V_{\mathrm{thin}}]
   \le\binom m{m-\ell}2^N(1-p)^{\beta ns/2}.
  \end{equation}

  \medskip\noindent\emph{Thick subspaces.}
  Let $M=\lfloor\beta n/2\rfloor$. Choose independent
  $Y_1,\ldots,Y_M$, independently of the $X_i$, each by selecting
  an index uniformly from $[m]\setminus E$ and taking an independent
  copy of the corresponding $Z_i$. For $V\in\mathcal V_{\mathrm{thick}}$,
  \eqref{eq:relative-consequence} gives
  \[
   \Prob[Y_j\in V]
   =\frac1N\sum_{i\notin E}\Prob[Z_i\in V]
   \ge\gamma^{-h/(2Q)}\rho_V.
  \]
  For fixed $V$, the event $\{V_X=V\}$ depends only on the $X_i$.
  The $Y_j$ are independent of these vectors and of one another,
  and adding them leaves the span equal to $V$ if they all lie in $V$.
  Hence
  \begin{equation}\label{eq:thick-lower}
   \Prob[U+\operatorname{span}(X_1,\ldots,X_m,Y_1,\ldots,Y_M)=V]
   \ge\gamma^{-hM/(2Q)}\rho_V^M\Prob[V_X=V].
  \end{equation}

  If the combined span is $V$, choose $J\subseteq[M]$ so that the
  $Y_j$, $j\in J$, extend a basis of $U$ to a basis of
  $U+\operatorname{span}(Y_1,\ldots,Y_M)$. Then choose
  $I\subseteq[m]$ so that the $X_i$, $i\in I$, extend this basis
  to $V$. Thus $|I|+|J|=m-\ell$, and, writing
  \[
   W_J=U+\operatorname{span}(Y_j:j\in J),\qquad
   V_{I,J}=W_J+\operatorname{span}(X_i:i\in I),
  \]
  we have $V_{I,J}=V$, every $Y_j$ with $j\notin J$ lies in
  $W_J$, and every $X_i$ with $i\notin I$ lies in $V$.
  Fix $I,J$ and reveal only $(Y_j)_{j\in J}$ and $(X_i)_{i\in I}$.
  We do not condition on these vectors forming a basis.
  The other vectors remain independent. On $V_{I,J}=V$,
  the bound on $r_i(U)$ and $|J|\le M$ give
  \[
   r_i(W_J)\ge\beta n-|J|\ge\beta n/2
   \qquad(i\notin E).
  \]
  Hence~\eqref{eq:absolute-consequence} gives, after this reveal,
  $\Prob[Y_j\in W_J]\le(1-p)^{\beta n/2}<\rho_V$
  for $j\notin J$. Also,
  $|[m]\setminus(I\cup E)|\ge s+|J|$, so
  \eqref{eq:maclaurin-product-bound} bounds the probability that
  the $X_i$ with $i\notin I\cup E$ lie in $V$ by
  $2^N\rho_V^{s+|J|}$.
  The conditional probability of both requirements is therefore at
  most $2^N\rho_V^{M+s}\le2^N\rho_V^M$.
  Only values of the revealed vectors with $V_{I,J}=V$ contribute.
  Averaging over the selected vectors and summing over $I,J$ yields
  \[
   \Prob[U+\operatorname{span}(X_1,\ldots,X_m,Y_1,\ldots,Y_M)=V]
   \le2^N\rho_V^M
   \sum_{\substack{I\subseteq[m],\ J\subseteq[M]\\
                    |I|+|J|=m-\ell}}
         \Prob[V_{I,J}=V].
  \]
  Since $\rho_V>0$, compare this with~\eqref{eq:thick-lower}
  and sum over $V\in\mathcal V_{\mathrm{thick}}$.
  For fixed $I,J$, the events $\{V_{I,J}=V\}$ are disjoint, and
  there are $\binom{m+M}{m-\ell}$ choices of $I,J$. Therefore
  \begin{equation}\label{eq:thick-bound}
   \Prob[V_X\in\mathcal V_{\mathrm{thick}}]
   \le2^N\gamma^{hM/(2Q)}\binom{m+M}{m-\ell}.
  \end{equation}

  By~\eqref{eq:fixed-skeleton-input}, $q_*<\ell\le h$.
  Using $N\le m$, $s\ge h/4$, and $M\ge\beta n/4$ for large $n$,
  we sum~\eqref{eq:thin-bound} and~\eqref{eq:thick-bound} over
  at most $m^{q_*}$ classes to obtain
  \[
  \begin{aligned}
   &\Prob[\codim(V_X)=h,\ V_X^\perp
        \cap\mathrm{Sparse}(\delta)=\{0\}]\\
   &\qquad\le m^{q_*}\left(
       2^{2m}(1-p)^{\beta nh/8}
      +2^{2m+M}\gamma^{\beta nh/(8Q)}\right).
  \end{aligned}
  \]
  Since $m^{q_*}\le n^h$ and $m,M=O(n)$, taking
  $k_{\ref{prop:fixed-skeleton}}$ large and then $n$ large makes
  the right-hand side at most $e^{-c_{\ref{prop:fixed-skeleton}}nh}$.
  \end{proof}

  \subsection{Proof of Proposition~\ref{cor:nonsparse-tail}}
  \label{sec:nonsparse}

  \begin{proof}[Proof of Proposition~\ref{cor:nonsparse-tail}]
  Return to the original law. Increase
  $k_{\ref{cor:nonsparse-tail}}$ so that
  Proposition~\ref{prop:fixed-skeleton} applies and $h\ge8Q$.
  Fix an integer
  $k_{\ref{cor:nonsparse-tail}}\le h<
  k_{\ref{lem:global-conditioning}}\log n$, and let
  $\mathcal N_h$ be the event that $\corank(A)=h$ and
  $\ker(A)\cap\mathrm{Sparse}(\delta)=\{0\}$. Let
  $q_*=N_\theta+Q(\mathrm D(A,\pi)+1)$.
  By~\eqref{eq:capture-inequality},
  \begin{equation}\label{eq:exact-capture}
   \Prob_A[\mathcal N_h]
   \le n^{\lceil h/2\rceil}
   \Prob_{A,\pi}[\mathcal N_h,\ \ell\ge\lceil h/2\rceil].
  \end{equation}
  On the event on the right, if $q_*>\ell/2$, then
  \[
   N_\theta+Q\mathrm D(A,\pi)>h/4-Q\ge h/8.
  \]
  Thus $N_\theta\ge\lceil h/16\rceil$ or
  $\mathrm D(A,\pi)\ge\lceil h/(16Q)\rceil$.
  Both integers are positive and at most
  $h<k_{\ref{lem:global-conditioning}}\log n$, so they lie in the
  range of Proposition~\ref{prop:defect-tails}.
  Proposition~\ref{prop:defect-tails} bounds the probability of
  their union by
  $4e^{-c_{\ref{prop:defect-tails}}nh/(16Q)}$.

  The values of $\ell,q_*$ and the event in
  Lemma~\ref{lem:transversality} are determined by
  $\mathcal G_\pi$. Outside the event in
  Lemma~\ref{lem:transversality}, if $q_*\le\ell/2$ and
  $\ell\ge\lceil h/2\rceil$, then
  $\ell-q_*\ge\ell/2\ge h/4$.
  Moreover, $V_X$ is the row space of $A$, and
  $V_X^\perp=\ker(A)$. Proposition~\ref{prop:fixed-skeleton}
  therefore gives a conditional bound
  $e^{-c_{\ref{prop:fixed-skeleton}}nh}$ for $\mathcal N_h$.
  Averaging over the matching and the pair sums, we obtain
  \[
   \Prob_{A,\pi}[\mathcal N_h,\ \ell\ge\lceil h/2\rceil]
   \le4e^{-c_{\ref{prop:defect-tails}}nh/(16Q)}
     +e^{-c_{\ref{lem:transversality}}n^2}
     +e^{-c_{\ref{prop:fixed-skeleton}}nh}.
  \]
  Since $h\le n$, the term $e^{-c_{\ref{lem:transversality}}n^2}$
  is at most $e^{-c_{\ref{lem:transversality}}nh}$.
  Also, $n^{\lceil h/2\rceil}\le e^{h\log n}$ and
  $h\log n=o(nh)$. Thus~\eqref{eq:exact-capture} gives,
  for some $c>0$ depending
  only on $\lambda$ and all sufficiently large $n$,
  \begin{equation}\label{eq:exact-nonsparse}
   \Prob_A[\mathcal N_h]\le e^{-cnh}.
  \end{equation}

  For $k\ge k_{\ref{lem:global-conditioning}}\log n$,
  the proposition follows from~\eqref{eq:large-corank-conditioning}.
  For $k_{\ref{cor:nonsparse-tail}}\le k<
  k_{\ref{lem:global-conditioning}}\log n$, splitting according to
  the exact corank and applying~\eqref{eq:exact-nonsparse} gives
  \[
  \begin{aligned}
   &\Prob[\corank(A)\ge k,\
        \ker(A)\cap\mathrm{Sparse}(\delta)=\{0\}]\\
   &\quad\le
   \sum_{h=k}^{\lceil k_{\ref{lem:global-conditioning}}\log n\rceil-1}
         e^{-cnh}
   +e^{-c_{\ref{lem:global-conditioning}}n
         \lceil k_{\ref{lem:global-conditioning}}\log n\rceil}\\
   &\quad\le
   \left(\frac1{1-e^{-cn}}+1\right)
         e^{-\min\{c,c_{\ref{lem:global-conditioning}}\}nk}.
  \end{aligned}
  \]
  We have $1/(1-e^{-cn})+1\le3$ for large $n$.
  Taking $c_{\ref{cor:nonsparse-tail}}$ smaller than half the
  minimum in the last exponent gives the required bound, since
  $3\le e^{c_{\ref{cor:nonsparse-tail}}nk}$ for all sufficiently
  large $n$.
  \end{proof}

  \begingroup
  \section*{Statement on the Use of Artificial Intelligence}
  The main ideas, results, and proof strategy of this paper were
  developed by the authors. During the preparation of the manuscript,
  ChatGPT and Codex (OpenAI), including the GPT-5.6 Sol and GPT-6 Pro
  models, were used as auxiliary tools to check and refine some technical
  details of the row pairing and thin--thick arguments after the underlying
  ideas and proof structure had been supplied by the authors, and to assist
  with language editing. All AI-assisted material included in the manuscript
  was independently reviewed and verified by the authors, who take full
  responsibility for the mathematical content and final text.
  \endgroup

  \begingroup
  \section*{Acknowledgments}
  This work began with a discussion between H. Liang and K. Yu.\@
  H. Liang and K. Yu are also grateful to Zengjing Chen and Hanchao Wang
  for making it possible for them to discuss this work in detail at the seminar.
  T. Yu is grateful to his advisor, Alexander Litvak, for his continued
  support, encouragement, and many helpful suggestions.
  \endgroup

  \bibliographystyle{amsplain}
  \bibliography{ref}

  \vspace{1.2cm}

  \par\bigskip
  \begingroup
  \small
  \setlength{\parindent}{0pt}

  \textsc{Huaijin Liang}\par\nobreak
  \smallskip
  Institute for Financial Studies,
  Shandong University, Jinan 250100, China.\par\nobreak
  \smallskip
  \textit{Email address:} \texttt{lianghuaij@gmail.com}\par

  \bigskip

  \textsc{Kexin Yu}\par\nobreak
  \smallskip
  School of Mathematics, Shandong University,
  Jinan 250100, China.\par\nobreak
  \smallskip
  \textit{Email address:} \texttt{kexinyu01@gmail.com}\par

  \bigskip

  \textsc{Tingzhou Yu}\par\nobreak
  \smallskip
  Department of Mathematical and Statistical Sciences,
  University of Alberta, Edmonton, AB T6G 2R3, Canada.\par\nobreak
  \smallskip
  \textit{Email address:} \texttt{tingzho1@ualberta.ca}\par

  \endgroup

  \end{document}